\documentclass[]{amsart}

\usepackage[margin=1in]{geometry}
\usepackage{amsmath,amssymb,amsthm,mathtools}
\usepackage{enumitem}
\usepackage{hyperref}
\hypersetup{
  colorlinks=true,
  linkcolor=blue,
  citecolor=blue,
  urlcolor=blue
}

\newtheorem{thm}{Theorem}[section]

\newtheorem*{fixed point criterion}{\fixed point criterion}
\newtheorem{cor}[thm]{Corollary}
\newtheorem{lem}[thm]{Lemma}
\newtheorem{prop}[thm]{Proposition}

\theoremstyle{definition}
\newtheorem{defn}[thm]{Definition}

\theoremstyle{remark}
\newtheorem{rem}[thm]{Remark}
\numberwithin{equation}{section}

\newcommand{\SL}{\operatorname{SL}}

\newcommand{\Z}{\mathbb Z}

\newcommand{\rank}{\operatorname{rank}}
\newcommand{\Stab}{\operatorname{Stab}}

\newcommand{\rk}{\mathrm{rk}}

\title{ $\SL(3,\Z)$ is not Howson}

\thanks{The authors are partially supported by NSFC (No. 12471066 and 11971389).}



\date{\today}

\author{Shengkui Ye}
\address{NYU Shanghai, No.567 Yangsi West 
  Rd, Pudong New Area, Shanghai, 200124, P.R. China \\
NYU-ECNU Institute of Mathematical Sciences at NYU Shanghai, 3663 Zhongshan Road North, Shanghai, 200062, China}
\email{sy55@nyu.edu}

\author{Qiang Zhang}
\address{School of Mathematics and Statistics, Xi’an Jiaotong University, Xi’an 710049, China}
\email{zhangq.math@mail.xjtu.edu.cn}

\begin{document}

\maketitle

\begin{abstract}
We give an explicit construction of two $2$-generated subgroups
$H,K\leq \SL(3,\Z)$ whose intersection is not finitely generated. The
construction takes place inside the standard parabolic subgroup
$\Z^2\rtimes \SL(2,\Z)\leq \SL(3,\Z)$. The main point is to identify
$H\cap K$ with the stabilizer of a point for an affine action of a free group
on $\Z^2$, and then to prove, using the Schreier graph of this action, that
this stabilizer is not finitely generated. Furthermore, we prove that there exists a sequence of subgroups $H_q, K_q \leq \SL(3,\mathbb{Z})$ such that $\rank(H_q)=\rank(K_q)=4$, and
\[
\rank(H_q\cap K_q)\geq q+1,
\] while $H_q\cap K_q$ is finitely generated.
\end{abstract}

\section{Introduction}
Let $G$ be a group.  The \emph{rank} of $G$, denoted
$\rank(G)$, is the minimal cardinality of a generating set for $G$.

In 1954,  Howson \cite{How54} showed that the intersection of finitely generated subgroups of a free group is also finitely generated, which led to the notion of Howson group: a group $G$ is said to be \textit{Howson}, if the intersection $H\cap K$ of any two finitely generated subgroups $H, K<G$ is again finitely generated. Many types of groups are Howson, for instance, free groups, surface groups, Baumslag-Solitar groups $BS_{1,n}=\langle a,t \mid tat^{-1}=a^n \rangle  (n\geq 1)$ \cite{Mol68}, limit groups  \cite{Dah03}, etc. In particular, if $G$ is a free or surface group, then
$$\rank(H\cap K)-1\leq (\rank(H)-1)(\rank(K)-1),$$
which was conjectured by Hanna Neumann in 1957, and proved independently by Friedman \cite{Fri14} and by Mineyev \cite{Min12} for free groups in 2011, and by Antol\'{i}n and Jaikin-Zapirain \cite{AJZ22} for surface groups in 2022. Note that the direct product $F_n \times \Z$ of a free group $F_n$ with rank $n>1$ is not Howson. Moreover, Kapovich \cite{Kap97} showed many hyperbolic groups are not Howson.

It is easy to show that the class of Howson groups is closed under taking subgroups and under ﬁnite extensions. In particular, virtually free groups and virtually surface groups are both Howson. More generally, the class of Howson groups is closed under graphs of groups, where the edge groups are ﬁnite  \cite{Sky05}. 
Moreover, Shusterman and Zalesskii \cite{SZ20} extended the Howson property to the Demushkin group: the intersection of a pair of closed topologically finitely generated subgroups of a Demushkin group is topologically finitely generated. This gives the first example of the Howson pro-$p$ groups that are not free.


It is known that the special linear group $\SL(2,\mathbb{Z})$ is Howson because it is virtually free. In contrast, $\SL(n,\Z) (n\geq 4)$ is not Howson \cite{LR11}, since it contains $F_2 \times \mathbb{Z}$. Nevertheless, it still remains unclear whether $\SL(3,\Z)$ is Howson (see \cite{LR11} and \cite{ka}). In this note we prove that $\SL(3,\Z)$ is not Howson: the intersection of two
$2$-generated subgroups can fail to be finitely generated.

\begin{thm}\label{thm sl3n not howson}
The group $\SL(3,\Z)$ is not Howson.
\end{thm}

In 2015, Ara\'ujo, Silva and Sykiotis \cite{Sky15} introduced a stronger quantitative version. A group $G$ is called \emph{strongly Howson} if, for all
positive integers $m,n$, there exists a constant $C(m,n)$ such that whenever
$H,K\leq G$ satisfy $\rank(H)\leq m$ and $\rank(K)\leq n$, 
we have
\[
    \rank(H\cap K)\leq C(m,n).
\]
Thus strong Howson is a uniform rank-bounded version of the Howson property. In \cite{ZZ261}, Zhang and Zhao showed that the strongly Howson property is indeed stronger than the Howson property, by constructing the first examples of Howson groups which are not strongly Howson. In this note, we have:

\begin{thm}\label{SL3n not satisfy strong Howson-type uniform rank}
Inside $\SL(3,\Z)$ there exist finitely generated subgroups $H_q,K_q$ for any integer $q\geq 2$ such
that
$\rank(H_q)=\rank(K_q)=4,$
while $H_q\cap K_q$ is finitely generated and
\[
\rank(H_q\cap K_q)\geq q+1.
\]
In particular, $\SL(3,\Z)$ does not satisfy any strong Howson-type uniform rank
bound, even when restricted to those pairs whose intersections are finitely
generated.
\end{thm}

\section{Schreier graphs: basic facts}

We shall use only standard facts about Schreier graphs of subgroups of free
groups.  See Stallings \cite{Stallings1983} for the core-graph viewpoint and
Lyndon--Schupp \cite[Chapter I]{LyndonSchupp1977} for classical combinatorial
group theory background.

Let $F=F(X)$ be the free group on a finite set $X$, and let $N\leq F$ be a
subgroup.

\begin{defn}[Schreier graph]
The \emph{(right) Schreier graph} $\Gamma(N\backslash F,X)$ has vertex set
$N\backslash F$, the set of right cosets.  For every vertex $Ng$ and every generator $x\in X$, there is
a directed edge
   $Ng \xrightarrow{x} Ngx$.
Equivalently, one may include inverse edges labeled $x^{-1}$.
\end{defn}

The following properties are the ones used in the proof of our main theorems.

\begin{enumerate}[label=\textbf{S\arabic*.}, leftmargin=2.4em]
\item \textbf{Connectedness.}
The Schreier graph $\Gamma(N\backslash F,X)$ is connected.  Indeed, every
vertex $Ng$ is reached from the base vertex $N$ by reading a word representing
$g$.

\item \textbf{Words and paths.}
A word in the alphabet $X^{\pm1}$ labels a path in the Schreier graph.  If the
word is freely reduced, then the corresponding path has no immediate
backtracking.

\item \textbf{Loops and the subgroup.}
A word $w\in F$ labels a loop at the base vertex $N$ if and only if $w\in N$.
More generally, $w$ labels a loop at the vertex $Ng$ if and only if
$gwg^{-1}\in N$, with the convention appropriate to right cosets.

\item \textbf{Orbital interpretation.}
If $F$ acts transitively on a set $\Omega$ and $\omega_0\in \Omega$, then the
Schreier graph of $\Stab_F(\omega_0)$ is naturally isomorphic to the orbital
graph of the action on $\omega_0 F$ (The orbital graph has vertex set $\Omega=\omega_0 F$, and two vertices $\omega_0 g,\omega_0 gx$ are connected by an edge for $x\in X$).  Under this isomorphism, the base vertex
corresponds to $\omega_0$.

\item \textbf{The core.}
The \emph{core} of a Schreier graph is the subgraph spanned by all edges and
vertices lying on reduced closed paths.  Equivalently, it is obtained by
pruning off all hanging trees.  A vertex that lies on a nontrivial reduced
closed path lies in the core.

\end{enumerate}
Some form of the following is known to Stallings \cite{Stallings1983}.

\begin{lem}[Stallings core criterion]\label{Stallings core
criterion}
    For a subgroup $N\leq F(X)$, the subgroup $N$ is finitely generated if and only
if the core of $\Gamma(N\backslash F,X)$ is finite.  In particular, if the core
contains infinitely many vertices, then $N$ is not finitely generated.
\end{lem}
\begin{proof}
    A word $w\in F$ labels a loop at the base vertex $N$ in the graph $\Gamma(N\backslash F,X)$  if and only if $w\in N$. Therefore, the subgroup $N$ is the set $P$ of labels of (reduced) closed paths based at the vertex $N$.  Note that $P$ can be viewed as a group (isomorphic to the group $N$) with the product defined as the concatenation of labels of paths. Therefore, $N$ is finitely generated if and only if the core is finite.
\end{proof}

\begin{rem}
The Schreier graph is also referred to as the relative Cayley graph in some literature, for example, see \cite[Section 2]{ZZ262}. 
 
Given a left action of a group $G$ on a set $X$, a right action can be defined via $x\cdot g:=g^{-1}\cdot x$. Note that the stabilizers of a point in $X$ are isomorphic for these two actions. While Schreier graphs and Cayley graphs usually adopt right multiplication, matrix actions adopt left multiplication. Following standard conventions, we adopt left group actions on sets throughout the subsequent two sections. The proofs of Theorems \ref{thm sl3n not howson} and \ref{SL3n not satisfy strong Howson-type uniform rank} rely on analyzing the corresponding stabilizers.
\end{rem}

\section{Proof of Theorem \ref{thm sl3n not howson}}

\subsection{Two $2$-generated subgroups of $\SL(3,\Z)$}
Consider the standard parabolic subgroup
\begin{equation}\label{eq. P}
  P=
\left\{
\begin{pmatrix}
A & v\\
0 & 1
\end{pmatrix}
\mid A\in \SL(2,\Z),\ v\in \Z^2
\right\}
\leq \SL(3,\Z).  
\end{equation}
We identify $P$ with the semidirect
product $\Z^2\rtimes \SL(2,\Z)$, with multiplication
\[
(v,A)(v',A')=(v+Av',AA').
\]
Moreover,
$\begin{pmatrix}
A & v\\
0 & 1
\end{pmatrix}
$
acts on column vectors $x\in \Z^2$ as the affine transformation
$x\longmapsto Ax+v.$ 
Let
\begin{equation}\label{eq. U V}
 U=\begin{pmatrix}1&2\\0&1\end{pmatrix},
\qquad
V=\begin{pmatrix}1&0\\2&1\end{pmatrix}.   
\end{equation}
By Sanov's theorem \cite{Sanov1947}, $U$ and $V$ freely generate a free subgroup
\begin{equation}\label{eq. F}
 F=\langle U,V\rangle\cong F_2   
\end{equation}
of $\SL(2,\Z)$. Define $H=\langle u,v\rangle$, where
\[
u=
\begin{pmatrix}
U&0\\
0&1
\end{pmatrix},
\qquad
v=
\begin{pmatrix}
V&0\\
0&1
\end{pmatrix}.
\]
Thus $H$ is the copy of $F$ sitting linearly in the upper-left block:
\[
H=
\left\{
\begin{pmatrix}
w & 0\\
0 & 1
\end{pmatrix}
\mid w\in F
\right\}
\cong F_2.
\]

Next define affine lifts of $U$ and $V$ by
\begin{equation}\label{eq. hatu hatv}
\hat u=
\begin{pmatrix}
1&2&0\\
0&1&1\\
0&0&1
\end{pmatrix},
\qquad
\hat v=
\begin{pmatrix}
1&0&1\\
2&1&0\\
0&0&1
\end{pmatrix}, 
\end{equation}
and put
\[
    K=\langle \hat u,\hat v\rangle.
\]
The projection
\[
    P=\Z^2\rtimes \SL(2,\Z)\longrightarrow \SL(2,\Z)
\]
sends $\hat u$ to $U$ and $\hat v$ to $V$.  Since $U,V$ freely generate
$F$, this projection restricts to an isomorphism $K\cong F$. Therefore $K$ is also freely generated by two elements.

\subsection{The intersection as a stabilizer}

The matrices $\hat u$ and $\hat v$ act on $\Z^2$ as the affine maps
\begin{equation}\label{eq. alapa beta}
  \alpha(x,y)=(x+2y,\ y+1),\quad   \beta(x,y)=(x+1,\ 2x+y).
\end{equation}

For a word $w\in F=\langle U,V\rangle$, let $c(w)\in \Z^2$ denote the
translation part of the corresponding element of $K$.  Thus the element of $K$
with linear part $w$ has the form
\[
\begin{pmatrix}
w & c(w)\\
0 & 1
\end{pmatrix}.
\]
The corresponding element of $H$ has translation part $0$.  Therefore,
\[
    H\cap K=\{w\in F \mid c(w)=0\}:=N.
\]

Since the affine transformation associated to $w$ sends $(0,0)$ to $c(w)$, this
subgroup can be described as
\[
    N=\Stab_F(0,0)
\]
for the affine action of $F$ on $\Z^2$ generated by $\alpha$ and $\beta$.
Thus proving that $H\cap K$ is not finitely generated is equivalent to proving
that this stabilizer $N$ is not finitely generated.

\subsection{Infinitely many vertices in the Schreier core}

For $n\in \Z$, define
\begin{equation}\label{eq. Pn}
   P_n=(n,1-n)\in \Z^2. 
\end{equation}
These points lie on the affine line $L=\{(n,1-n)\mid n\in \Z\}.$

\begin{lem}\label{lem:orbit-infinitely-many}
The orbit of $(0,0)$ under the affine action generated by $\alpha$ and $\beta$
contains all the points $P_n$ for $n\in \Z$.
\end{lem}

\begin{proof}
First,
\[
    \alpha(0,0)=(0,1)=P_0,
\]
and
\[
    \beta(0,0)=(1,0)=P_1.
\]
For every integer $m$, a direct calculation gives
\[
    \alpha^m(x,y)=\bigl(x+2my+m(m-1),\ y+m\bigr),
\]
and
\[
    \beta^m(x,y)=\bigl(x+m,\ y+2mx+m(m-1)\bigr).
\]
Applying these formulas with negative exponents gives, for $n\geq 0$,
\[
\beta^{-2n}(P_n)=P_{-n},
\qquad
\alpha^{-2n-2}(P_{-n})=P_{n+2}.
\]
Thus, starting from $P_0$ and $P_1$, we obtain successively all points
$P_n$ with all $n\in \Z$.
\end{proof}

\begin{lem}\label{lem:loops}
Every point $P_n$ lies on a nontrivial reduced closed path in the orbital
Schreier graph of the affine action.
\end{lem}

\begin{proof}
Consider the reduced word
\[
    r=(U^{-1}V)^2=U^{-1}VU^{-1}V\in F.
\]
It is nontrivial because $U,V$ freely generate $F$.  The corresponding affine
transformation is $(\alpha^{-1}\beta)^2$.
For $P_n=(n,1-n)$, one checks that
$\alpha^{-1}\beta(P_n)=P_{1-n}.$
Therefore
\[
    (\alpha^{-1}\beta)^2(P_n)=P_n
\]
for every $n\in \Z$.  Hence the word $r$ labels a nontrivial reduced closed
path at every vertex $P_n$.
\end{proof}

\begin{prop}\label{prop:N-not-fg}
The stabilizer $N=\Stab_F(0,0)$ is not finitely generated.
\end{prop}

\begin{proof}
By Lemma~\ref{lem:orbit-infinitely-many}, the orbit of $(0,0)$ contains infinitely
many distinct vertices $P_n$.  By Lemma~\ref{lem:loops}, each of these vertices lies
on a nontrivial reduced closed path.  Therefore the core of the orbital
Schreier graph contains infinitely many vertices.

By the orbital interpretation of Schreier graphs, this orbital graph is the
Schreier graph of the subgroup $N=\Stab_F(0,0)$.  By the Stallings core
criterion (Lemma \ref{Stallings core criterion}), a subgroup of a finitely generated free group is finitely generated
if and only if its Schreier core is finite.  Since the core here is infinite,
$N$ is not finitely generated.
\end{proof}

\subsection{$\SL(3,\Z)$ is not Howson.}
\begin{proof}[Proof of Theorem \ref{thm sl3n not howson}]
We have constructed two subgroups
\[
    H=\langle u,v\rangle,
    \qquad
    K=\langle \hat u,\hat v\rangle
\]
of $\SL(3,\Z)$, both generated by two elements, such that
\[
    H\cap K\cong N=\Stab_F(0,0).
\]
By Proposition~\ref{prop:N-not-fg}, $N$ is not finitely generated.  Therefore
$H\cap K$ is not finitely generated. Hence $\SL(3,\Z)$ is not a Howson group.
\end{proof}

In fact, $\SL(3,\Z)$ fails to be a Howson group, as it contains
$\Z^2\rtimes \SL(2,\Z)\geq \Z^2\rtimes F_2$. The proof actually proves the following result.

\begin{cor}
    Both $\Z^2\rtimes F_2$ and $\Z^2\rtimes \SL(2,\Z)$ are not Howson.
\end{cor}

\section{Proof of Theorem \ref{SL3n not satisfy strong Howson-type uniform rank}}

In this section, we shall construct, inside $\SL(3,\Z)$, two families of finitely generated subgroups
$H_q$ and $K_q$ such that $\rank(H_q)=\rank(K_q)=4$,
while $H_q\cap K_q$ is finitely generated, but
\[
\rank(H_q\cap K_q)\to \infty.
\]

\subsection{The subgroups $H_q$ and $K_q$}
We take $P$ as in Eq. (\ref{eq. P}) and
\[
U=\begin{pmatrix}1&2\\0&1\end{pmatrix},
\qquad
V=\begin{pmatrix}1&0\\2&1\end{pmatrix}
\]
as in Eq. (\ref{eq. U V}). Then as in Eq. (\ref{eq. F}), $ F=\langle U,V\rangle\leq \SL(2,\Z)$ is free of rank $2$.  We shall work inside
\[
\Z^2\rtimes F\leq P\leq \SL(3,\Z).
\]
Let
\[
e_1=\begin{pmatrix}1\\0\end{pmatrix},
\qquad
 e_2=\begin{pmatrix}0\\1\end{pmatrix}.
\]
Define affine lifts of $U$ and $V$ by
\[
\hat u=(e_2,U),
\qquad
\hat v=(e_1,V).
\]
In matrix form,
\[
\hat u=
\begin{pmatrix}
1&2&0\\
0&1&1\\
0&0&1
\end{pmatrix},
\qquad
\hat v=
\begin{pmatrix}
1&0&1\\
2&1&0\\
0&0&1
\end{pmatrix},
\]
are as in Eq. (\ref{eq. hatu hatv}).

Fix an integer $q\geq 2$, and put
\[
L_q=q\Z^2.
\]
Define
\begin{equation}\label{eq. Hq Kq}
\begin{array}{llcll}
H_q&=&\langle L_q,(0,U),(0,V)\rangle
&\leq& \Z^2\rtimes F,\\
K_q&=&\langle L_q, ~\hat u, ~\hat v\rangle
&\leq& \Z^2\rtimes F.
\end{array}
\end{equation}
Equivalently, in $\SL(3,\Z)$,
\[
H_q=
\left\langle
\begin{pmatrix}1&0&q\\0&1&0\\0&0&1\end{pmatrix},
\begin{pmatrix}1&0&0\\0&1&q\\0&0&1\end{pmatrix},
\begin{pmatrix}1&2&0\\0&1&0\\0&0&1\end{pmatrix},
\begin{pmatrix}1&0&0\\2&1&0\\0&0&1\end{pmatrix}
\right\rangle,
\]
and
\[
K_q=
\left\langle
\begin{pmatrix}1&0&q\\0&1&0\\0&0&1\end{pmatrix},
\begin{pmatrix}1&0&0\\0&1&q\\0&0&1\end{pmatrix},
\begin{pmatrix}1&2&0\\0&1&1\\0&0&1\end{pmatrix},
\begin{pmatrix}1&0&1\\2&1&0\\0&0&1\end{pmatrix}
\right\rangle.
\]

\begin{lem}
For every $q\geq 2$, we have $\rank(H_q)=\rank(K_q)=4.$
\end{lem}

\begin{proof}
Both groups are generated by the four displayed generators, so their ranks are
at most $4$. Furthermore, we have
\[
H_q\cong L_q\rtimes F,
\]
where $F=\langle U,V\rangle\cong F_2$ acts on $L_q=q\Z^2$ by the standard linear
action.  The abelianization of this semidirect product is
\[
(H_q)_{ab}\cong
(F)_{ab}\oplus
\frac{L_q}{(U-I)L_q+(V-I)L_q}.
\]
Now
\[
(U-I)L_q=2q\Z e_1,
\qquad
(V-I)L_q=2q\Z e_2.
\]
Hence
\[
\frac{L_q}{(U-I)L_q+(V-I)L_q}
\cong
(\Z/2\Z)^2.
\]
Since $F_{ab}\cong \Z^2$, we get
\[
(H_q)_{ab}\cong \Z^2\oplus (\Z/2\Z)^2.
\]
This abelian group requires four generators.  Therefore $H_q$ itself requires
at least four generators, so $\rank(H_q)=4$.

The same argument applies to $K_q$.  Indeed, because $F$ is free on $U,V$, the
assignment
\[
U\mapsto \hat u,
\qquad
V\mapsto \hat v
\]
defines a free subgroup of $\Z^2\rtimes F$ projecting isomorphically onto $F$.
Thus
\[
K_q\cong L_q\rtimes F
\]
with the same linear action.  Consequently
\[
(K_q)_{ab}\cong \Z^2\oplus (\Z/2\Z)^2,
\]
and $\rank(K_q)=4$.
\end{proof}

\subsection{Description of the intersection}

Let $c:F\longrightarrow \Z^2$
be the $1$-cocycle determined by
\[
c(U)=e_2,
\qquad
c(V)=e_1.
\]
Explicitly, we have
\[
K_q=\{(z+c(w),w)\mid z\in L_q,\\ w\in F\}.
\]
On the other hand,
\[
H_q=\{(z,w)\mid z\in L_q,
\ w\in F\}.
\]
Therefore
\[
H_q\cap K_q=\{(z,w)\mid z\in L_q, w\in F,
\ c(w)\in L_q\},
\]
whose image in $F$ under the epimorphism $\phi: \Z^2\rtimes F \rightarrow F$ is 
\[
N_q=\{w\in F\mid c(w)\in L_q\}.
\]

Note that $\hat u$ and
$\hat v$ act on $(\Z/q\Z)^2$ (as a modulo $q$ quotient of $\Z^2$) by the affine transformations
\[
\bar\alpha(x,y)=(x+2y,y+1),
\]
and
\[
\bar\beta(x,y)=(x+1,2x+y).
\]
Then $N_q$ is the stabilizer of $(0,0)$ for this action.  Since
$(\Z/q\Z)^2$ is finite, $N_q$ has finite index in $F$.  By the
Nielsen--Schreier theorem, $N_q$ is a finitely generated free group; see, for
example, \cite[Ch.~I, Sec.~3]{LyndonSchupp1977}.  Hence $H_q\cap K_q$ is
finitely generated.

\subsection{The ranks of the intersections are unbounded}
Take $P_n=(n,1-n)\in \Z^2$ as in Eq. (\ref{eq. Pn}). Then by Lemma \ref{lem:orbit-infinitely-many}, all the points $P_n$ are contained in the orbit of $(0,0)$.
Reducing modulo $q$, the points
\[
P_0,P_1,\ldots,P_{q-1}
\]
are distinct in $(\Z/q\Z)^2$, because their first coordinates are distinct
modulo $q$.  Therefore the orbit of $(0,0)$ in $(\Z/q\Z)^2$ has at least $q$
points.  Since $N_q$ is the stabilizer of $(0,0)$, the orbit-stabilizer formula
for the action of $F$ gives
\[
[F:N_q]\geq q.
\]
Because $F\cong F_2$, the Nielsen--Schreier formula gives
\[
\rank(N_q)=[F:N_q](\rank(F)-1)+1=[F:N_q]+1\geq q+1.
\]
Finally, since the projection
\[
\phi: H_q\cap K_q \longrightarrow N_q
\]
is onto, one has
\[
\rank(H_q\cap K_q)\geq \rank(N_q)\geq q+1.
\]
The groups $H_q\cap K_q$ are finitely generated, but their ranks are
unbounded. Therefore, Theorem \ref{SL3n not satisfy strong Howson-type uniform rank} holds. 

In fact, from Eq. (\ref{eq. Hq Kq}), we have proved the following result.

\begin{thm}
For any integer $q\geq 2$, there exist finitely generated subgroups $H_q,K_q$ of $\Z^2\rtimes F_2$ such
that
\[
\rank(H_q)=\rank(K_q)=4,
\]
while $H_q\cap K_q$ is finitely generated and
\[
\rank(H_q\cap K_q)\geq q+1.
\]
In particular, $\Z^2\rtimes F_2$ does not satisfy any strong Howson-type uniform rank
bound, even when restricted to those pairs whose intersections are finitely
generated.
\end{thm}

\subsection*{Acknowledgments and AI disclosure}
ChatGPT was used to generate the potential examples that appear in this paper. The final proofs in this paper were reviewed and corrected by the human authors.

\end{document}